\pdfoutput=1
\documentclass{amsart}
\usepackage{
 amsmath,
 amssymb,
 stmaryrd,
 mathtools,
 multicol,
 enumerate,
 tikz-cd
}

\usepackage[pagebackref]{hyperref}

\DeclareMathOperator{\GL}{GL}
\DeclareMathOperator{\GSp}{GSp}

\DeclareMathOperator{\SL}{SL}
\DeclareMathOperator{\SO}{SO}
\DeclareMathOperator{\GU}{GU}

\DeclareMathOperator{\pr}{pr}
\DeclareMathOperator{\Res}{Res}

\DeclareMathOperator{\ord}{ord}

\renewcommand{\AA}{\mathbf{A}}

\newcommand{\ZZ}{\mathbf{Z}}
\newcommand{\RR}{\mathbf{R}}
\newcommand{\QQ}{\mathbf{Q}}
\newcommand{\Zp}{\ZZ_p}
\newcommand{\Gm}{\mathbf{G}_m}

\newcommand{\fg}{\mathfrak{g}}
\newcommand{\fh}{\mathfrak{h}}
\newcommand{\fsl}{\mathfrak{sl}}
\newcommand{\so}{\mathfrak{so}}
\newcommand{\fsp}{\mathfrak{sp}}

\newcommand{\cG}{\mathcal{G}}
\newcommand{\cH}{\mathcal{H}}
\newcommand{\cO}{\mathcal{O}}

\newcommand{\cX}{\mathcal{X}}
\newcommand{\cY}{\mathcal{Y}}

\newcommand{\Mi}{M_{\Iw}}
\newcommand{\Iw}{\mathrm{Iw}}
\newcommand{\et}{\text{\textup{\'et}}}

\newcommand{\into}{\hookrightarrow}

\renewcommand{\ge}{\geqslant}

\newcommand{\teq}{\trianglelefteqslant}

\newcommand{\stbt}[4]{
 \left(\begin{smallmatrix}#1 & #2 \\ #3 & #4\end{smallmatrix}\right)
}

\newtheorem{definition}{Definition}[subsection]
\newtheorem{theorem}[definition]{Theorem}
\newtheorem{proposition}[definition]{Proposition}
\newtheorem{lemma}[definition]{Lemma}

\newtheorem*{hyp*}{Hypothesis}

\theoremstyle{remark}
\newtheorem{remark}[definition]{Remark}
\newtheorem*{notation}{Notation}

\begin{document}

 \title[Spherical varieties and norm relations]{Spherical varieties and norm relations\\ in Iwasawa theory}
 \author{David Loeffler}
 \address{David Loeffler\\
  Mathematics Institute, University of Warwick, Coventry CV4 7AL, UK.\\
 \emph{ORCID}: \href{http://orcid.org/0000-0001-9069-1877}{\texttt{0000-0001-9069-1877}}}
 \email{d.a.loeffler@warwick.ac.uk}
 \urladdr{\url{http://warwick.ac.uk/fac/sci/maths/people/staff/david_loeffler/}}



 \thanks{Supported by a Royal Society University Research Fellowship.}

 \maketitle


 \begin{abstract}
  Norm-compatible families of cohomology classes for Shimura varieties, and other arithmetic symmetric spaces, play an important role in Iwasawa theory of automorphic forms. The aim of this note is to give a systematic approach to proving ``vertical'' norm-compatibility relations for such classes (where the level varies at a fixed prime $p$), treating the case of Betti and \'etale cohomology at once, and revealing an unexpected relation to the theory of spherical varieties. This machinery can be used to construct many new examples of norm-compatible families, potentially giving rise to new constructions of both Euler systems and $p$-adic $L$-functions: examples include families of algebraic cycles on Shimura varieties for $U(n) \times U(n+1)$ and $U(2n)$ over the $p$-adic anticyclotomic tower.
 \end{abstract}

 \section{Introduction}

  \subsection{Setting} The goal of this paper is to study norm-compatible families of cohomology classes attached to arithmetic symmetric spaces. Perhaps the simplest non-trivial example is the \emph{Mazur--Tate elements}, appearing in the theory of modular symbols for $\GL_2 / \QQ$. These are the elements defined by
  \[ \Theta_{m, N} \coloneqq \sum_{a \in (\ZZ / m\ZZ)^\times} [a] \otimes \{ \tfrac{a}{m} \to i\infty\}\ \ \in\ \  \ZZ[(\ZZ / m)^\times] \otimes_{\ZZ} H^1(Y_1(N), \ZZ), \]
  where $m, N \ge 1$ are integers, $Y_1(N)$ is the modular curve of level $\Gamma_1(N)$, and $\{ \tfrac{a}{m} \to i\infty\}$ denotes the image in $Y_1(N)$ of a path from $\tfrac{a}{m}$ to $i \infty$ in the complex upper half-plane. These Mazur--Tate elements satisfy the following crucial \emph{vertical norm relation}\footnote{There is also a ``horizontal'' norm relation, relating $\Theta_{\ell m, N}$ and $\Theta_{m, N}$ where $\ell$ is a prime not dividing $mN$, but we shall not discuss horizontal norm relations here.}: if $p$ is a prime dividing $N$, and $r \ge 1$, then
  \begin{equation}
   \label{eq:msnorm}
  \operatorname{norm}_{p^r}^{p^{r+1}}\left(\Theta_{p^{r+1}, N}\right)= U_p' \cdot \Theta_{p^r, N},
  \end{equation}
  where ``norm'' denotes the projection $\ZZ[(\ZZ / p^{r + 1})^\times] \to \ZZ[(\ZZ / p^r)^\times]$, and $U_p'$ is the transpose, with respect to Poincar\'e duality, of the usual $U_p$ operator.  This norm-compatibility relation is the crucial input in constructing the $p$-adic $L$-function of a weight 2 modular form. These elements also satisfy a norm-compatibility property in $N$, which is the input needed to extend the $p$-adic $L$-function to a 2-variable $p$-adic $L$-function for a Hida (or more generally Coleman) families of modular forms.

  A second, apparently rather different, setting in which ``norm-compat\-ibility'' problems arise is the theory of Euler systems. The \emph{Beilinson--Flach elements}, defined in \cite{LLZ}, are classes
  \begin{equation}
   \label{eq:bfnorm}
   \mathcal{BF}_{m, N} \in H^3_\mathrm{mot}\Big( (Y_1(N) \times Y_1(N))_{\QQ(\mu_m)}, \ZZ(2) \Big),
  \end{equation}
  where ``$\mathrm{mot}$'' denotes motivic cohomology. These turn out to satisfy norm-compatibility relations (in both $m$ and $N$) which are formally very similar to those of the Mazur--Tate elements; and these are crucial in applications of the Beilinson--Flach elements to Iwasawa theory and the Bloch--Kato conjecture.

  \subsection{Results of the paper} In this note, we develop a general formalism for proving vertical norm-compatibility relations for families of cohomology classes built up by pushing forward cohomology from a ``small'' reductive group $H$ to a ``large'' one $G$. The basic condition we need is that some subgroup of $H$ (depending on the setting) should act with an open orbit on a flag variety for $G$, and that the stabiliser of a point in this orbit should be as small as possible. This links our approach with the theory of \emph{spherical varieties}, which are precisely those $G$-varieties $G/H$ in which $H$ has an open orbit on the Borel flag variety of $G$.

  For instance, we can intepret the norm-compatibility \eqref{eq:msnorm} as a consequence of the fact that the subgroup $\stbt \star {}{} 1 \subseteq \GL_2$ has an open orbit on $\mathbf{P}^1$ with trivial stabiliser; and we can interpret \eqref{eq:bfnorm} as a consequence of the more subtle statement that the subgroup $\stbt * * {} 1 \subseteq \GL_2$, embedded diagonally inside $\GL_2 \times \GL_2$, has an open orbit on $\mathbf{P}^1 \times \mathbf{P}^1$ with trivial stabiliser.

  Our construction is entirely local at $p$, and applies to any cohomology theory satisfying a list of straightforward properties. This gives simple, uniform proofs of a wide range of norm-compatibility statements appearing in the literature on $p$-adic $L$-functions and Euler systems. More importantly, it also gives rise to many new results.

  One special case of our construction is the following theorem. Let $\cG$ be a reductive group over $\QQ$, and $\cH \subseteq \cG$ a reductive subgroup, equipped with compatible Hodge-type Shimura data. For simplicity, we assume that the centres of $\cG$ and $\cH$ have no isogeny factor which is $\RR$-split but not $\QQ$-split. Let $\mathcal{C}$ be the maximal torus quotient of $\cH$, and $E$ the reflex field of the Shimura datum for $\cH$, so the action of Galois on the connected components of the Shimura variety $\cY_{\cH}$ is given by the composite of the Artin reciprocity map for $E$ with a homomorphism $\Res_{E/\QQ}(\mathbf{G}_m) \to \mathcal{C}$. Let $G, H, C$ denote the base-extensions of $\cG, \cH, \mathcal{C}$ to $\QQ_p$, for some prime $p$ such that $G,H,C$ are unramified.

  \begin{theorem}
   \label{thm:cycles}
   Suppose that there exists a parabolic subgroup $Q_G$ of $G$, with opposite $\overline{Q}_G$, and a point $u \in (G / \overline{Q}_G)(\QQ_p)$, such that:
   \begin{itemize}
    \item the $H$-orbit of $u$ is Zariski-open in $G / \overline{Q}_G$;
    \item the image in $C$ of the subgroup $H \cap u \overline{Q}_G u^{-1}$ (the $H$-stabiliser of $u$) is a proper subgroup $C^0 \subset C$.
   \end{itemize}
   Then the $Q_G$-ordinary projections of cycle classes of $\cY_{\cH}$ in $\cY_{\cG}$ interpolate into an Iwasawa cohomology class over the abelian $p$-adic Lie extension of $E$ corresponding to $C^0$.
  \end{theorem}

  As instances of this, we obtain norm-compatible families of cycles (in the arithmetic middle degree) over the $p$-adic anticyclotomic tower for Shimura varieties attached to the groups $U(n, 1) \times U(n-1, 1)$ and $U(2n-1, 1)$, for any $n$.

  \subsection*{Acknowledgements} I would like to thank Antonio Cauchi, Christophe Cornut, Andrew Graham, Dimitar Jetchev, Jan Nekov\v{a}\'r, Aaron Pollack, Joaquin Rodrigues Jacinto, Shrenik Shah, Chris Skinner, Chris Williams and Sarah Livia Zerbes for illuminating discussions and comments in connection with this paper. I am also grateful to Syed Waqar Ali Shah for several useful comments on an earlier draft (in particular Remark \ref{rmk:waqar} below), and to the anonymous referee for his or her close reading of the paper.

 \section{Formalism of cohomology functors}

  We begin by introducing a formalism which is intended to model the behaviour of cohomology of symmetric spaces attached to reductive groups. This section is entirely abstract nonsense; its aim is to allow the theorems of the later parts of this paper to be stated and proved in a uniform way, by axiomatising the properties that a ``reasonable'' cohomology theory should satisfy. (The real work in this paper will begin at \S \ref{sect:norm-compat}.)

  \subsection{Cohomology functors}

  Let $G$ be a locally pro-finite topological group, and $\Sigma \subseteq G$ an open submonoid (not necessarily a subgroup). We shall frequently omit to specify $\Sigma$, in which case it should be understood that $\Sigma = G$. We let $\Sigma^{-1}$ be the monoid $\{ g^{-1}: g \in \Sigma\}$.

  \begin{definition}
   We let $\mathcal{P}(G, \Sigma)$ be the category whose objects are the open compact subgroups of $G$ contained in $\Sigma$, and whose morphisms are given by
   \[ \operatorname{Hom}_{\mathcal{P}(G, \Sigma)}(U, V) = \{ g \in U \backslash \Sigma / V: g^{-1} U g \subseteq V\}. \]
   We write $[g]_{U, V}$, or just $[g]$, for the morphism $U \to V$ corresponding to the double coset of $g$, with composition defined by $[g] \circ [h] = [hg]$ for any two composable morphisms $[g], [h]$. If $U \subseteq V$ we write $[1]_{U, V}$ as $\operatorname{pr}_{U, V}$ or just $\operatorname{pr}$.
  \end{definition}

  \begin{definition}
   By a \emph{cohomology functor} $M$ for $(G, \Sigma)$ (with coefficients in some commutative ring $A$) we mean a pair of functors $M = (M_\star, M^\star)$, where
   \[ M^\star: \mathcal{P}(G, \Sigma)^{\mathrm{op}} \to A\text{-}\mathrm{Mod} \quad\text{and}\quad
   M_\star: \mathcal{P}(G, \Sigma^{-1}) \to A\text{-}\mathrm{Mod},\]
   such that:
   \begin{enumerate}
    \item $M_\star(U) = M^\star(U)$ for every object $U$; we write the common value simply as $M(U)$.
    \item writing $[g]_{U, V, \star} = M_\star([g]_{U, V})$ and similarly $[g]^\star_{U, V}$, we have
    \[ [g]_{U, V}^\star = [g^{-1}]_{V, U, \star} \in \operatorname{Hom}_{A}(M(V), M(U))\]
    whenever this makes sense, i.e. whenever $g^{-1}U g = V $ and $g \in \Sigma$.
   \end{enumerate}
   We shall omit the subscripts $U, V$ if they are clear from context. A \emph{morphism of cohomology functors} $M \to N$ is a collection of $A$-module maps $M(U) \to N(U)$ for each open compact $U \subseteq \Sigma$, compatible with the maps $[g]_\star$ and $[g]^\star$.
  \end{definition}

  In practice, we shall obtain examples as follows: we shall consider a functor from $\mathcal{P}(G)$ to some category of geometric objects (e.g.~manifolds or schemes) sending $U$ to a ``symmetric space of level $U$'', and the maps $[g]$ will correspond to degeneracy maps between these symmetric spaces, twisted by the right-translation action of $g \in G$. Taking cohomology of these spaces -- for any ``reasonable'' cohomology theory, admitting pushforward and pullback maps -- will then give a cohomology functor in the above sense. This will be made precise in \S \ref{sect:examples} below. The role of the monoid $\Sigma$ is to allow us to work with cohomology with coefficients in lattices in $G$-representations (which may not be invariant under the whole group $G$).

  \begin{definition}
   \label{def:cart}
   We say $M$ is \emph{Cartesian} if the following condition is satisfied: for any open compact subgroup $V \subseteq \Sigma$ and any two open compact subgroups $U, U' \subseteq V$, we have a commutative diagram
   \[
    \begin{tikzcd}[row sep=large]
     \bigoplus_\gamma M(U_\gamma)
     \arrow[r, "\sum \pr_\star"]
     & M(U)\\
     M(U')
     \arrow[r, "\pr_\star"']
     \arrow[u, "{\sum [\gamma]^\star}"]
     & M(V)
     \arrow[u, "\pr^\star"']
    \end{tikzcd}
   \]
   where the sum runs over the double quotient $\gamma \in U \backslash V / U'$; we define $U_\gamma = \gamma U' \gamma^{-1} \cap U$; the left vertical map is the sum of the pullback maps for the inclusions $\gamma^{-1} U_\gamma \gamma \subset U'$; and the top horizontal map is the direct sum of the natural pushforward maps.
  \end{definition}

  Note that if $U \triangleleft V$ and we take $U'= U$, then all the $U_\gamma$ are equal to $U$, and we see that the composite of pushforward and pullback has to be given by summing over coset representatives for $V / U$.

  \begin{remark}
   \label{rmk:waqar}
   The above notion of a Cartesian cohomology functor is the analogue for locally profinite groups of the notion of a \emph{Mackey functor} considered in the representation theory of finite or profinite groups; see \cite{webb} for an overview of this theory. (I am very grateful to S.W.A.~Shah for this observation.)
  \end{remark}

 \subsection{Completions}

  If $M$ is a cohomology functor for $(G, \Sigma)$, then there are two canonical ways to extend $M(-)$ from open compact subgroups to all compact subgroups, which correspond roughly to the ``completed cohomology'' and ``completed homology'' of Emerton \cite{emerton06}. We define
  \[ \overline{M}(K) = \varinjlim_{U \supseteq K} M(U),\qquad M_{\mathrm{Iw}}(K) = \varprojlim_{U \supseteq K} M(U)\]
  where the limits run over open compact subgroups of $\Sigma$ containing $K$. Evidently, the first limit is taken with respect to the pullback maps $\pr^\star$, and the second with respect to the pushforwards $\pr_\star$.

  We shall not use $\overline{M}$ in the present paper (we mention it only for completeness); it is $M_{\mathrm{Iw}}$ which is most relevant. We shall refer to it as the \emph{Iwasawa completion}, by analogy with Iwasawa cohomology groups of $p$-adic Galois representations (in fact this is more than a mere analogy, as we shall see in due course).

  It is clear that we can define pushforward maps $[g]_*: \Mi(K) \to \Mi(K')$ for all triples $(K, K', g)$ with $g^{-1} K g \subseteq K'$ and $g \in \Sigma^{-1}$ (compatibly with the given definition when $K, K'$ are open). More subtly, if $M$ is Cartesian, we can also define \emph{finite} pullback maps on Iwasawa cohomology: if $K' \subseteq K$ has finite index, then we can find systems of open compact subgroups $(K_n)_{n \ge 1}$ with $\bigcap_{n \ge 1} K_n = K$, and similarly $(K_n')_{n \ge 1}$ with intersection $K'$, such that $K_n' \cap K = K'$ and $[K_n : K_n'] = [K : K']$ for all $n$. The Cartesian property then implies that pullback maps from level $K_n$ to level $K_n'$ are compatible with the pushforwards for changing $n$, and we deduce the existence of pullback maps at the infinite level making the following diagram commute for all $n$:
  \[
   \begin{tikzcd}
    \Mi(K) \arrow[r] \arrow[d] & \Mi(K')\arrow[d]\\
    M(K_n) \ar[r]& M(K_n').
   \end{tikzcd}
  \]
  With these definitions, the pushforwards and pullbacks in Iwasawa cohomology satisfy a Cartesian property extending Definition \ref{def:cart}, for any three compact subgroups $U, U' \subseteq V$ with $[V : U] < \infty$.

  \begin{remark}
   The situation for $\overline{M}$ is, of course, exactly the opposite: one can define finite pushforwards and arbitrary pullbacks.
  \end{remark}

 \subsection{Functoriality in $G$}

  Let $\iota: H \into G$ be the inclusion of a closed subgroup, and $M_G$ a Cartesian cohomology functor for $(G, \Sigma)$. Then we can define a cohomology functor $\iota^{!}(M_G)$ for $(H, \Sigma \cap H)$ by setting
  \[ \iota^{!}(M_G)(U) = M_{G, \Iw}(\iota(U)). \]

  \begin{definition}
   If $M_H$ and $M_G$ are Cartesian cohomology functors for $H$ and $G$ respectively, then a \emph{pushforward map} $\iota_\star: M_H \to M_G$ is a morphism $M_H \to \iota^!(M_G)$ of cohomology functors for $(H, \Sigma \cap H)$.
  \end{definition}

  A more pedestrian definition is that a pushforward map consists of morphisms $\iota_{U, \star}: M_H(U \cap H) \to M_G(U)$ for any open compact $U \subseteq \Sigma$, compatible with pushforward maps $[h]_\star$ for $h \in H \cap \Sigma^{-1}$, and satisfying a compatibility with pullbacks expressed in terms of a Cartesian diagram involving the double quotient $U \backslash V / (V \cap H)$.

\section{Examples of cohomology functors}

 \label{sect:examples}
 The motivating examples of the above formalism arise as follows.

 \subsection{Betti cohomology}

  Let us suppose that $G = \cG(\QQ_p)$, where $\cG$ is a connected reductive group over $\QQ$. Let $K_\infty^\prime$ denote a maximal compact subgroup of $\cG'(\RR)^\circ$, where $\cG'$ is the derived subgroup of $\cG$, and $(-)^\circ$ denotes the identity component. We set $K_\infty = K_\infty' \cdot \mathcal{Z}(\RR)$, where $\mathcal{Z}$ is the centre of $\cG$, and $\mathcal{X} = \cG(\RR) / K_\infty$; this has natural structure as a smooth manifold, preserved by the left action of $\cG(\RR)$. (Alternatively, we can define $K_\infty$ as the preimage in $G(\RR)$ of a maximal compact subgroup of $\cG^{\mathrm{ad}}(\RR)^{\circ}$, where $\cG^{\mathrm{ad}} = \cG / \mathcal{Z}$.)

  We choose an open compact subgroup $U^p \subseteq \cG(\AA_{\mathrm{f}}^p)$ which is small enough that for \emph{any} open compact $U \subseteq G$, the product $U^p U$ is neat. Such subgroups exist, since there are only finitely many conjugacy classes of maximal compacts in $G$. We can then define
  \[ Y(U) = \cG(\QQ) \backslash \Big( (\cG(\AA_\mathrm{f}) / U^p U) \times \mathcal{X}\Big), \]
  Our assumptions on $U^p$ imply that $Y(U)$ is a smooth manifold for every $U$, and the right action of $G$ on $\cG(\AA)$ gives a covariant functor
  \[ Y: \mathcal{P}(G) \to \underline{\text{Man}}^{\mathrm{ur}}, \]
  where $\underline{\text{Man}}^{\mathrm{ur}}$ is the category whose objects are smooth manifolds and whose morphisms are finite unramified coverings. Since Betti cohomology is both covariantly and contravariantly functorial on $\underline{\text{Man}}^{\mathrm{ur}}$, we obtain cohomology functors $M_G(-) = H^i(Y(-), A)$, for every $i \ge 0$ and ring $A$. These are not in general Cartesian, so we shall impose the following hypothesis:

  \begin{hyp*}[``Axiom SV5'']
   The centre $\mathcal{Z}$ is isogenous to the product of a $\QQ$-split torus and an $\RR$-anisotropic torus. Equivalently, $\mathcal{Z}(\QQ)$ is discrete in $\mathcal{Z}(\AA_\mathrm{f})$.
  \end{hyp*}

  (See the section ``Additional axioms'' in \cite[Chapter 5]{milne}, where the widely-used abbreviation SV5 originates).

  It follows from SV5 that for any morphism $[g]: U \to V$ in $\mathcal{P}(G)$, the map $Y(U) \to Y(V)$ has degree $[V: g^{-1} U g]$. In the setting of Definition \ref{def:cart}, we have $[V: U] = \sum_{\gamma} [U': U_\gamma]$. It follows that in the commutative square
  \[
   \begin{tikzcd}[row sep=large]
    \bigsqcup_{\gamma} Y(U_\gamma) \ar[r, "{\sqcup_\gamma \pr}"]
    \ar[d, "{\sqcup_\gamma [\gamma]}"'] &
    Y(U)\ar[d, "\pr"] \\
    Y(U') \ar[r, "\pr"] &Y(V)
   \end{tikzcd}\tag{\dag}
  \]
  all the maps are surjections and the vertical maps have the same degree, and hence the diagram is Cartesian. Since pushforward and pullback maps commute in Cartesian diagrams, we deduce that $H^i(Y(-), A)$ is a Cartesian cohomology functor. So we have shown:

  \begin{proposition}
   If Axiom SV5 holds, then for any ring $A$ and integer $i \ge 0$, the functor $M(-) = H^i(Y(-), A)$ is a Cartesian cohomology functor for $G$ with coefficients in $A$ (and similarly for compactly-supported cohomology).\qed 
  \end{proposition}

  \begin{remark}
   There are many important examples where Axiom SV5 is not satisfied, such as Hilbert modular groups $\Res_{K/\QQ} \GL_2$ for $K$ totally real. These can be dealt with by the following workaround. Let $E$ denote the closure in $\cG(\QQ_p)$ of the discrete group $\mathcal{Z}(\QQ) \cap U^p Z_0$, where $Z_0$ is the (unique) maximal compact of $\mathcal{Z}(\QQ_p)$. We then set $G = \cG(\QQ_p) / E$, and for $U \subseteq G$, we let $Y(U)$ be the symmetric space of level $U^p \cdot \pi^{-1}(U)$, where $\pi$ is the projection from $\cG(\QQ_p)$ to $G$. Then the cohomology of $Y(-)$ is Cartesian as a functor on open compacts of $G$. We leave the details to the interested reader.
  \end{remark}

  \subsection{Coefficients}

   More generally, we may also consider cohomology with coefficients in local systems; here the role of the monoid $\Sigma$ becomes important. If Axiom SV5 holds, and $M$ is an $A$-module with an $A$-linear left action of $\Sigma$, then for every open compact $U \subseteq \Sigma$, the $U$-action on $M$ gives rise to a local system $\mathcal{V}_M$ of $A$-modules on $Y(U)$; and for every morphism $[g]: U \to V$ in $\mathcal{P}(G, \Sigma)$, we can define $[g]^*$ to be the composite
  \[ H^i(Y(V), \mathcal{V}_M) \to H^i(Y(U), [g]^* \mathcal{V}_M) \to H^i(Y(U), \mathcal{V}_M)\]
  where the second arrow is given by the action of $g$ on $M$. The same construction gives pushforward maps for morphisms in $\mathcal{P}(G, \Sigma^{-1})$; so the groups $H^i(Y(-), \mathcal{V}_M)$ form a cohomology functor for $(G, \Sigma)$, and one can verify that this is also Cartesian.

  One obvious case of interest is when $A = \mathcal{O}_K$, for some $p$-adic field $K$, and $M$ is an $\mathcal{O}_K$-lattice in an algebraic representation of $\cG$ over $K$. In this case, one can take $\Sigma$ to be the monoid $\{ g \in G: g \cdot M \subseteq M\}$. However, one can also consider more sophisticated coefficient modules (not necessarily of finite type over $A$), such as the modules of locally analytic distributions appearing in \cite{loefflerzerbes16} and \cite{JLZ}.

 \subsection{\'Etale cohomology}

  Let us now suppose that $\cG$ admits a Shimura datum, so that we can identify $\mathcal{X}$ with the set of $\cG(\RR)$-conjugates of a cocharacter $h: \Res_{\mathbf{C}/\RR} \GL_1 \to \cG_{\RR}$ satisfying the axioms of \cite{deligne71}. Then Deligne's theory of canonical models shows that there is a number field $E \subset \mathbf{C}$ (the reflex field) and a smooth quasiprojective $E$-variety
  \[ \mathcal{Y}(U) := \operatorname{Sh}_{U^p U}(\cG, \mathcal{X})_{E}, \]
  for each open $U$, whose $\mathbf{C}$-points are canonically identified with $Y(U)$.

  \begin{proposition}
   For any integers $i, n$, the groups
   \[ M(U) = H^i_{\et}\Big(\mathcal{Y}(U), \Zp(n)\Big)\]
   form a Cartesian cohomology functor with coefficients in $\Zp$; and similarly for motivic cohomology with coefficients in $\ZZ$.
  \end{proposition}

  This is proved much as before: $\mathcal{Y}(-)$ becomes a functor from $\mathcal{P}(G)$ to the category of smooth $E$-varieties and \'etale coverings, and \'etale cohomology is both covariantly and contravarantly functorial on this category. Moreover, Axiom SV5 implies that the diagram corresponding to (\dag) is Cartesian in the category of $E$-varieties (not just topological spaces) so one obtains the Cartesian property of the cohomology from this.

  One can also replace $\mathcal{Y}(U)$ with its canonical integral model over $\mathcal{O}_{E, S}$, for $S$ a sufficiently large finite set of places (containing all those above $p$), in situations where such models are known to exist (e.g.~if the Shimura datum $(\cG, \mathcal{X})$ is of Hodge type); this has the advantage that the \'etale cohomology groups become finitely-generated over $\ZZ_p$. We can also consider \'etale cohomology with coefficients, much as in the Betti theory above.

 \subsection{Functoriality}

  Now suppose that we have an embedding $\iota: \cH \into \cG$ of reductive groups over $\QQ$ (both satisfying Axiom SV5). We can then apply the constructions above to either $\cH$ or $\cG$, and we indicate the group concerned by a subscript.

  If there is a cocharacter $h: \Res_{\mathbf{C}/\RR} \GL_1 \to \cH_{\RR}$ such that $(\cH, [h])$ and $(\cG, [\iota \circ h])$ are both Shimura data, then $K_{\cH, \infty}$ (resp. $K_{\cG, \infty}$) is identified with the centraliser of $h$ in $\cH(\RR)$ (resp. $\cG(\RR)$). It follows that $K_{\cH, \infty} = K_{\cG, \infty} \cap \cH$, so that $\cX_{\cH}$ is a closed submanifold of $\cX_{\cG}$.

  For more general embeddings of groups $\cH \into \cG$ (not necessarily underlying a morphism of Shimura data), there may not be a natural map $\cX_\cH \to \cX_\cG$, because $K_{\cG, \infty} \cap \cH(\RR)$ can be strictly smaller than $K_{\cH, \infty}$. We thus define
  \[ \widetilde{\cX}_{\cH} = \cH(\RR) / \left( K_{\cG, \infty} \cap \cH(\RR)\right), \]
  so that there are maps $\cX_\cH \twoheadleftarrow \widetilde{\cX}_{\cH} \into \cX_{\cG}$, compatible with the action of $\cH(\RR)$. Of course, in the Shimura variety setting we have $\widetilde{\cX}_{\cH} = \cX_{\cH}$; in the general case, $\widetilde{\cX}_{\cH} \to \cX_{\cH}$ is a real vector bundle.

  Finally, we shall choose a (sufficiently small) prime-to-$p$ level group $U_\cG^p$ for $\cG$, and define $U_{\cH}^p = \cH(\AA_{\mathrm{f}}^p) \cap U_\cG^p$. We thus have maps
  \[ Y_\cH(U \cap H) \twoheadleftarrow \tilde Y_{\cH}(U \cap H) \to Y_{\cG}(U) \]
  for every open compact $U \subseteq G$, where $H = \cH(\QQ_p)$. Moreover, since the fibres of $\tilde Y_{\cH} \to Y_\cH$ are real vector spaces, pullback along this map gives an isomorphism in cohomology (and also for compactly-supported cohomology, up to a shift in degree). With these definitions, the following is elementary:

  \begin{proposition}
   Define cohomology functors by
   \[ M_H(-) = H^i(Y_\cH( -), A),\qquad M_G(-) = H^{i + c}(Y_\cG(-), A), \]
   for some $i \ge 0$ and some coefficient ring $A$, where $c = \dim \cX_\cG - \dim \tilde \cX_{\cH}$. Then the composite of pullback and pushforward along the maps above defines a morphism of cohomology functors $\iota_\star: M_H \to M_G$. The same applies to \'etale cohomology, defining
   \[ M_H(-) = H^i_{\et}(\cY_\cH(-), \Zp(n)),\qquad M_G(-) = H^{i + 2c}_{\et}(\cY_\cG(-), \Zp(n+c)) \]
   for any $i, n$, where $c$ is the codimension of $\cX_\cH$ in $\cX_\cG$ as a complex manifold.
  \end{proposition}

  We can also formulate versions of these statements with coefficients, noting that the pullback of sheaves from $Y_{\cG}$ to $Y_{\cH}$ corresponds to restriction of modules from $\Sigma$ to $\Sigma \cap H$.

 \section{The norm-compatibility machine}
 \label{sect:norm-compat}

  \subsection{Notations} We suppose $\iota: H \into G$ is an inclusion of reductive group schemes over $\Zp$. We fix a Borel subgroup $B_G$ of $G$, and maximal torus $T_G \subseteq B_G$, in such a way that the intersections $B_H$, $T_H$ of these with $H$ are a Borel and maximal torus in $H$.

  \begin{definition}
   By a \emph{mirabolic subgroup} of $G$ we mean an algebraic subgroup-scheme $Q_G^0$ of the following form: we choose a parabolic $Q_G \supseteq B_G$, and let $Q_G = L_G \cdot N_G$ be its standard Levi factorisation (so $T_G \subseteq L_G$). We choose a normal subgroup $L_G^0 \teq L_G$, and we let $Q_G^0 = N_G \cdot L_G^0$. We define mirabolic subgroups of $H$ similarly.
  \end{definition}

  Our goal is to show that if $Q_G^0$ and $Q_H^0$ are mirabolics in $G$ and $H$ satisfying a certain compatibility property, then -- for any map of Cartesian cohomology functors $M_H \to M_G$ -- we obtain maps
  \[ M_{H, \Iw}(Q_H^0(\Zp)) \to \left[ M_{G, \Iw}(Q_G^{0}(\Zp))\right]^{\mathrm{fs}}, \]
  where ``fs'' denotes the finite-slope part for an appropriate Hecke operator (this will be defined below).

  \subsection{The input}
   \label{sect:Eisclass}

  As input, we need examples of ``interesting'' classes in $M_{H, \Iw}(Q_H^0(\Zp))$. Examples of these arise as follows:
  \begin{itemize}

   \item For any group $H$, and any globalisation $\cH$ of $H$, we can take $Q_H^0 = Q_H = H$, and consider the identity class in $M_H(H(\Zp))$ where $M_H$ denotes degree 0 Betti cohomology (or \'etale cohomology, when this is defined). Perhaps surprisingly, this case is by no means trivial, and will in fact give rise to many of our most interesting examples.

   \item For $\cH = \GL_2 / \QQ$ we can take $Q_H = \stbt**{}*$ the standard Borel and $Q^0_H$ the mirabolic subgroup $\stbt * * {} 1$. The Siegel units $\left({}_c g_{0, 1/p^r}\right)_{r \ge 1}$, for some suitable auxilliary integer $c > 1$, are a norm-compatible family of modular units of level $\{ Q_H^0 \bmod p^r\}$ (cf.~\cite[\S 2]{kato04}); they thus give rise to classes in $M_{H, \Iw}(Q^0_H(\Zp))$ with $M_H(-)$ taken to be degree 1 \'etale, Betti, or motivic cohomology.

   \item More generally, for $\cH = \GSp_{2n}$ we can take $Q_H$ to be the Klingen parabolic and $Q_H^0$ the ``mira-Klingen'' subgroup
   $\left(\begin{smallmatrix}
   \star&\star&\dots&\star&\star\\
    &\star&\dots&\star&\star\\
    &\vdots&\ddots&\vdots & \vdots\\
    &\star&\dots&\star&\star\\
    &&&&1
   \end{smallmatrix}\right)$. A construction due to Faltings \cite{faltings05} gives an integrally-normalised Eisenstein class with norm-compatibility in this tower, living in the groups $H^{2n-1}_{\et}(Y_{\cH}, \Zp(n))$; for $n = 1$ this reduces to Kato's Siegel unit class.

   \item We can take direct (i.e.~exterior cup) products of the above examples in the obvious fashion.
  \end{itemize}

  The norm-compatiblity satisfied by these classes is quite weak, or even vacuous (i.e. the quotients $H / Q_H^0$ are small or trivial). The machinery of this section will allow us to parlay this into a far stronger norm-compatiblity statement for their pushforwards to $G$.

  \subsection{A flag variety}

   Let $Q_H^0$ be a mirabolic in $H$, and $Q_G$ a parabolic in $G$. We consider the left action of $G$ on the quotient $\mathcal{F} =  G / \bar{Q}_G$, where $\bar{Q}_G$ is the opposite of $Q_G$ (relative to our fixed maximal torus $T_G$). Via the embedding $\iota$, we can restrict this to an action of $Q_H^0$. We shall assume there is some $u \in G(\Zp)$ such that the following conditions are satisfied:

   \begin{enumerate}[(A)]
    \item The $Q_H^0$-orbit of $u$ is open in $\mathcal{F}$.
    \item We have
    \[ u^{-1} Q_H^0 u \cap \bar{Q}_G \subseteq \bar{Q}_G^0, \]
    where $\bar{Q}_G^0 = \bar{N}_G \cdot L^0_G$ for some normal reductive subgroup $L^0_G \trianglelefteqslant L_G$.
   \end{enumerate}

   Of course, condition (B) is always satisfied if we take $L^0_G = L_G$, but we obtain stronger results if we take $L_G^0$ to be smaller. In most of the examples below, $L_G^0$ will turn out to be either $Z_G$ or $\{1\}$. We shall define $Q_G^0 = N_G \cdot L_G^0$ (the opposite of $\bar{Q}_G^0$).

 \subsection{Level groups}

  We fix some cocharacter $\eta \in X_{\bullet}(T_G)$ which factors through $Z(L_G)$, and which is strictly dominant, so that $\langle \eta, \Phi \rangle > 0$ for every relative root $\Phi$ of $G$ with respect to $Q_G$; and we set $\tau = \eta(p)$. (We shall show later that our constructions are actually independent of $\eta$.)

  We then have
  \[ \tau N_G(\Zp) \tau^{-1} \subseteq N_G(\Zp),\qquad  \tau^{-1}\bar{N}_G(\Zp) \tau \subseteq \bar{N}_G(\Zp),  \]
  and both $\tau N_G(\Zp) \tau^{-1}$ and $\tau^{-1}\bar{N}_G(\Zp) \tau$ are in the kernel of reduction modulo $p$ (so both inequalities are strict unless $N_G = \{1\}$, which is a trivial case). Moreover, if $\tau^{-r} g \tau^r \in G(\Zp)$ for some $r \ge 0$, then $g \pmod{p^r} \in \bar{Q}_G$.

  \begin{notation}
   We define the following open compact subgroups of $G(\Zp)$, for $r \ge 0$.
   \begin{itemize}
    \item $U_r \coloneqq \{ g \in G(\Zp): \hfill\tau^{-r} g \tau^r \in G(\Zp)$ and $ g \pmod{p^r} \in \bar{Q}_G^0\}$.
    \item $U_r' \coloneqq \{ g \in G(\Zp):\hfill \tau^{-(r+1)} g \tau^{(r+1)} \in G(\Zp)$ and $ g \pmod{p^r} \in \bar{Q}_G^0\}$.
    \item $V_r \coloneqq \tau^{-r} U_r \tau^r$.
   \end{itemize}
  \end{notation}

  Note that $U_r \supseteq U_r' \supseteq U_{r+1}$, and $U_r' = U_r \cap \tau U_r \tau^{-1}$. Moreover, for $r \ge 1$ all three groups have Iwahori decompositions with respect to $Q_G$ and $\bar{Q}_G$: if we set $N_r \coloneqq \tau^r N_G(\Zp) \tau^{-r}$, $\bar{N}_r \coloneqq \tau^{-r} \bar{N}_G(\Zp) \tau^{r}$, and $L_r \coloneqq \{ g \in L_G(\Zp): g \pmod{p^r} \in L^0_G\}$, then we have
  \begin{align*}
   U_r & = \bar{N}_0 \times L_r \times N_r, &
   U_r' &= \bar{N}_0 \times L_r\times N_{r+1}, &
   V_r &= \bar{N}_r \times L_r \times N_{0}.
  \end{align*}

  \begin{lemma} Suppose $r \ge 1$. Then:
   \begin{enumerate}[(i)]
    \item We have $u^{-1}Q_H^0 u \cap U_r'  = u^{-1}Q_H^0 u \cap U_{r+1} $.
    \item We have
    \[ [u^{-1}Q_H^0 u \cap U_{r}: u^{-1}Q_H^0 u \cap U_{r}'] = [U_r: U_r'].\]
   \end{enumerate}
  \end{lemma}

  \begin{proof}
   Part (i) follows from the assumption (B) on $Q_H^0$ (applied modulo $p^{(r+1)}$): if $q \in u^{-1}Q_H^0u \cap U_r'$, then $q \bmod p^{r+1}$ is in $ u^{-1}Q_H^0 u \cap \bar{Q}_G$, hence it is in $ u^{-1}Q_H^0 u  \cap \bar{Q}_G^0$.

   For part (ii), we need to show that there is a set of representatives for $U_r' \backslash U_r$ contained in $u^{-1}Q_H^0 u$. Projection to the $N$ factor of the Iwahori decomposition gives an isomorphism
   \[ U_r' \backslash U_r = N_{r+1} \backslash N_r. \]
   Let $[1]$ denote the image of the identity of $G$ in $\mathcal{F}$. Since the orbit of $[1]$ under $u^{-1}Q_H^0 u$ is open as a $\Zp$-subscheme of $\mathcal{F}$, and it contains $[1]$, it must contain the preimage in $\mathcal{F}(\Zp)$ of $[1] \bmod p$. Hence, for any $x \in N_r$, there exists $q \in Q_H^0(\Zp)$ such that $u^{-1} q u$ lies in $N_{r+1} x \bar{Q}_G$. In particular, $u^{-1} q u \pmod{p^r} \in \bar{Q}_G$, so in fact $u^{-1} q u \pmod{p^r} \in \bar{Q}_G^0$ and thus $u^{-1} q u \in U_r$. However, by construction $u^{-1} q u$ maps to the coset of $x$ in $N_{r+1} \backslash N_r$, so $u^{-1} q u$ represents the coset of $x$ in $U_r' \backslash U_r$.
  \end{proof}

 \subsection{The construction}

  Let $M_G$ be a Cartesian cohomology functor for $(G, \Sigma)$, where $\Sigma$ is some monoid containing $G(\Zp)$ and $\tau^{-1}$; let $M_H$ be a Cartesian cohomology functor for $(H, \Sigma \cap H)$; and let $\iota_\star: M_H \to M_G$ be a pushforward map. Then we consider the following diagram:
  \[
  \begin{tikzcd}
   M_{H, \Iw}(Q_H^0 \cap u U_{r + 1} u^{-1}) \arrow[r, "{[u]_\star}"] \arrow[d, equal] & M_G(U_{r+1}) \arrow[d] \arrow[rdd, dashed, out = -5, in=20, looseness=2.5, "{[\tau]_{U_{r+1}, U_r, \star}}" near start] &                                   \\
   M_{H, \Iw}(Q_H^0 \cap u U_r' u^{-1}) \arrow[r, "{[u]_\star}"]      & M_G(U_r') \arrow[r, "{[\tau]_\star}"]                                      & M_G(\tau^{-1} U_r'\tau) \arrow[d] \\
   M_{H, \Iw}(Q_H^0 \cap u U_r u^{-1}) \arrow[u] \arrow[r, "{[u]_\star}"]  & M_G(U_r) \arrow[u] \arrow[r, dotted]  & M_G(U_r)
  \end{tikzcd}
  \]
  Here upward (resp.~downward) vertical arrows are given by the pullback (resp.~pushforward) along the natural projection maps. The commutativity of the lower left square follows from assertion (ii) of the preceding lemma, together with the Cartesian axiom for $\iota_\star: M_H \to M_G$. The dotted arrow making the lower right square commute is the definition of the Hecke correspondence $\mathcal{T}$ associated to the double coset $[U_r \tau^{-1} U_r]$.

  Given $z_H \in M_H(Q_H^0(\Zp))$, we can define $z_{G, r}$ to be the image of $z_H$ under the map
  \[
   M_{H, \Iw}(Q_H^0) \xrightarrow{\pr^*} M_{H, \Iw}(Q_H^0 \cap u^{-1} U_{r} u) \xrightarrow{[u]_*} M_G(U_r).
  \]
  Note that this element depends only on the class of $u$ in $Q_H^0 \backslash G$.

  \begin{proposition}
   We have $[\tau]_{U_{r+1}, U_r, \star} (z_{G, r+1}) = \mathcal{T} \cdot z_{G, r}$.
  \end{proposition}

  \begin{proof}
   From the compatibility of the lower left square in the diagram, we know that $\pr_{U_{r+1}, U_r', \star}(z_{G, r+1}) = \pr_{U_r', U_r}^\star(z_{G, r})$ as elements of $M_G(U_r')$. The result now follows by mapping both sides into $M_G(U_r)$ via $[\tau]_\star$.
  \end{proof}

  This setup is convenient for proofs, but one can obtain tidier statements by replacing $U_r$ with its conjugate $V_r = \tau^{-r} U_r \tau^r$, and the classes $z_{G, r}$ with their cousins
  \[
   \xi_{G, r} =[\tau^r]_\star \cdot z_{G, r}  \in M_G(V_r).
  \]
  Then we have the following:

  \begin{proposition}
   We have $\pr_{V_{r+1}, V_r, \star}\left(\xi_{G, r+1}\right) = \mathcal{T} \cdot \xi_{G, r}$.\qed
  \end{proposition}

  It will be convenient to introduce the \emph{finite slope part}
  \[
   M_G(V_r)^{\mathrm{fs}} \coloneqq A[\mathcal{T}, \mathcal{T}^{-1}] \otimes_{A[\mathcal{T}]} M_G(V_r).
  \]

  \begin{proposition}
   The operators $\mathcal{T}$ on $M_G(V_r)$ and $M_G(V_{r+1})$ are compatible under the projection $\pr_{\star}: M_G(V_{r+1}) \to M_G(V_r)$, for each $r \ge 1$.
  \end{proposition}

  \begin{proof}
   From the Iwahori decomposition of $V_r$, we see that the operators $\mathcal{T}$ on $M_G(V_r)$ and $M_G(V_{r+1})$ admit a common set of coset representatives (given by any set of representatives for $N_0 / N_1$). So the compatibility follows from the Cartesian property of $M_G$.
  \end{proof}

  So we have a well-defined module
  \[ M_{G, \Iw}(Q_G^0(\Zp))^{\mathrm{fs}} = \varprojlim_r M_G(V_r)^{\mathrm{fs}}.\]

  \begin{theorem}
   \label{mainthm}
   The above construction defines a map
   \[ M_{H, \Iw}(Q_H^0(\Zp)) \to M_{G, \Iw}(Q_G^0(\Zp))^{\mathrm{fs}},\]
   mapping $z_H$ to the compatible system $\left(\mathcal{T}^{-r} \otimes \xi_{G, r}\right)_{r \ge 1}$. \qed
  \end{theorem}

 \subsection{Towers of ``abelian type''}
  \label{sect:abtowers}

  The above construction gives norm-compatibility in the towers $(V_r)_{r \ge 1}$, whose intersection is $Q_G^0(\Zp)$. This incorporates a huge amount of information. In practice it is often simpler to discard the ``non-abelian part'' of this information, by projecting to a simpler level tower.

  Let $\pi: H \to C$ be the maximal torus quotient of $H$ (as an algebraic group over $\QQ_p$), and let $\tilde G$ denote the direct product $G \times C$. The map $\tilde\iota = (\iota, \pi)$ gives a lifting of $H$ to a subgroup of $\tilde G$; and the second projection $\tilde G \to C$ gives an extension of $\pi$ to $\tilde G$.

  We suppose we have subgroups $Q_H^0$ and $Q_G$ satisfying the open-orbit condition (A), and such that $\pi\left(Q_H^0 \cap u\bar{Q}_Gu^{-1}\right)$ is contained in some subtorus $C^0 \subseteq C$. Note that this condition depends only on the $H$-orbit of $u$ in $(G/\overline{Q}_G)(\QQ_p)$. We can then replace $G$ with $\tilde G$ in all of the above constructions, and take
  \[ L_{\tilde G}^0 = L_G \times C^0 \subset L_{\tilde G} = L_G \times C.\]

  If $C_n$ denotes the preimage in $C(\Zp)$ of $C^0 \bmod p^n$, then the group $V_n \subset \tilde{G}(\Zp)$ arising from these new data is then contained (strictly if $n > 1$) in $J_0 \times C_n$, where $J \subseteq G(\Zp)$ is the parahoric subgroup associated to $Q_G$. Since both $V_n$ and $J$ have Iwahori decompositions with respect to the parabolic $Q_{\tilde G}$, the finite-slope parts for $\mathcal{T}$ are compatible under the projection $\pr_{V_n, J \times C_n, \star}$; so Theorem \ref{mainthm} gives us maps
  \[
   M_{H, \Iw}(Q_H^0(\Zp)) \to M_{\tilde G, \Iw}(J \times C_\infty)^{\mathrm{fs}},
  \]
  for any cohomology functor $M_{\tilde G}$ for $\tilde G$. In the setting of Betti and \'etale cohomology, these groups have a more ``classical'' interpretation, as we now show:

  \subsubsection*{The Betti setting: $p$-adic measures} Suppose that we are in the Betti cohomology setting, with $\Zp$-coefficients. That is, $G = \cG \times \QQ_p$ for some $\QQ$-group $\cG$, and $M_G(U) = H^i(Y_\cG(U^p U), \mathcal{V}_M)$ for some $i$, where $\mathcal{V}_M$ is the sheaf corresponding to a lattice $M$ in an algebraic representation of $G$ (preserved by $\tau^{-1}$).

  Let $\mathcal{C}$ be the maximal torus quotient of $\cH$, and let $\Delta_n$ denote the finite arithmetic quotient
  \[ \Delta_n = \mathcal{C}(\QQ) \backslash \mathcal{C}(\AA) /  C_n \cdot \mathcal{C}^p\cdot \mathcal{C}(\RR)^\dag, \]
  where $\mathcal{C}^p$ is the maximal compact subgroup of $\mathcal{C}(\AA_f^p)$, and $\mathcal{C}(\RR)^\dagger$ is the subset of the components of $\mathcal{C}(\RR)$ in the image of $K_{\cH, \infty}$. Then we obtain an extension of $M_G$ to a cohomology functor $M_{\tilde G}$ for $\tilde G$, such that
  \[
   M_{\tilde G}(J \times C_n)^{\mathrm{fs}} = \Zp[\Delta_n] \otimes_{\Zp} M_{G}(J)^{\mathrm{fs}}.
  \]
  Since $M_G(J)$ is finitely-generated over $\Zp$, the limit $e^{\ord} = \lim_{n \to \infty} \mathcal{T}^{n!}$ exists as an endomorphism of $M_G(J)$, and its image is the maximal direct summand $M_G(J)^{\ord}$ on which $\mathcal{T}$ is invertible. Thus we have a natural map $M_{\tilde G}(J \times C_n)^{\mathrm{fs}} \to M_G(J)^{\ord}$; so Theorem \ref{mainthm} gives a map
  \[ M_{H, \Iw}(Q_H^0) \to \Zp[[\Delta_\infty]] \otimes_{\Zp} H^i\left(Y_\cG(U^p J), \mathcal{V}_M\right)^{\ord}.\]
  The right-hand side is the space of $p$-adic measures on the abelian $p$-adic Lie group $\Delta_\infty$, with values in the $\Zp$-module $H^i\left(Y_\cG(U^p J), \mathcal{V}_M\right)^{\ord}$. So we have defined $p$-adic measures interpolating the $Q_G$-ordinary projections of classes pushed forward from $\cH$.

  \begin{remark}
   In most of the examples which interest us, $\mathcal{C}$ will just be $\mathbf{G}_m$, and $\mathcal{C}(\RR)^\dagger = \RR_{> 0}^\times$; and $C_m$ will be the principal congruence subgroup of level $m$, so that $\Delta_\infty \cong \Zp^\times$.
  \end{remark}

  \subsubsection*{The \'etale setting: Iwasawa cohomology} In the \'etale cohomology setting, we can proceed similarly, but we must now treat the finite groups $\Delta_n$ as 0-dimensional algebraic varieties over the reflex field $E$ of $\cY_\cH$, i.e.~as $\operatorname{Gal}(\overline{E} / E)$-sets. For simplicity we suppose that the Shimura datum for $G$ (and hence also for $H$) is of Hodge type; it follows that if $S$ is a sufficiently large finite set of places of $E$ (containing those above $p$), then all our Shimura varieties have smooth models over $\cO_{E, S}$, and the morphisms between them extend to the integral models.

  The Galois action on $\Delta_n$ is given by translation by a character
  \[ \kappa_n: \operatorname{Gal}(\overline{E} / E)^{\mathrm{ab}} \to \Delta_n, \]
  unramified outside $S$, which is the composite of the Artin map for $E$ and the cocharacter $\mu: \Res_{E/\QQ} \Gm \to \mathcal{C}$ determined by the Shimura datum.

  As before, the maps $\cY_{\cH}(U^p U \cap H) \to \cY_{\cG}(U^p U)$ lift to maps into the varieties
  \[ \cY_{\tilde \cG}(U^p U \times \mathcal{C}^p C_n) \cong \cY_{\cG}(U^p U) \times \Delta_n. \]
  The \'etale cohomology of these products over $\overline{E}$ is given by
  \[  H^i_{\et}\left(\cY_{\tilde \cG}(U^p U \times \mathcal{C}^p C_n)_{\overline{E}}, \mathcal{V}_M\right) \cong \Zp[\Delta_n](\kappa_n) \otimes_{\Zp} H^i_{\et}(\mathcal{Y}(U^p U)_{\overline{E}}, \mathcal{V}_M). \]
  This implies a spectral sequence\footnote{Here the use of the $S$-integral models is essential, in order to avoid issues involving compatibility with inverse limits.}
  \begin{multline*}
   E_2^{ij} = H^i\Big(\cO_{E, S}, \Zp[[\Delta_\infty]](\kappa_\infty) \otimes_{\Zp} H^j_{\et}(\mathcal{Y}_{\cG}(U^p J)_{\overline{E}}, \mathcal{V}_M)(n)\Big)\\
   \Rightarrow H^{i+j}_{\et, \Iw}\Big((\mathcal{Y}_{\cG}(U^p J) \times \Delta_\infty)_{\cO_{E, S}}, \mathcal{V}_M(n)\Big).
  \end{multline*}

  Let $\Gamma_\infty$ denote the image of $\kappa_\infty$ in $\Delta_\infty$. Then the $E_2^{ij}$ term can be written as
  \[ \Zp[[\Delta_\infty]] \otimes_{\Zp[[\Gamma_\infty]]} H^i_{\Iw}\Big(\cO_{E_\infty, S}, H^j_{\et}(\mathcal{Y}_{\cG}(U^p J)_{\overline{E}}, \mathcal{V}_M)(n)\Big)
  \]
  where $E_\infty$ is the abelian extension of $E$ (with Galois group $\Gamma_\infty$) cut out by the character $\kappa_\infty$. If $\Gamma_\infty$ has positive dimension, then the $E_2^{0j}$ terms vanish, so we obtain edge maps into $H^1_{\Iw}$. Thus Theorem \ref{mainthm} gives maps into the groups
  \[ \Zp[[\Delta_\infty]] \otimes_{\Zp[[\Gamma_\infty]]} H^1_{\Iw}\Big(\cO_{E_\infty, S}, H^{j-1}_{\et}(\mathcal{Y}_{\cG}(U^p J)_{\overline{E}}, \mathcal{V}_M)^{\ord}(n)\Big).\]

  Specialising to the case $Q_H^0 = H$ gives Theorem \ref{thm:cycles} of the introduction.

 \section{Some example cases}

  \subsection{Hecke-type pairs}

   We refer to pairs $(G, H)$ satisfying our conditions for $Q_H^0 = H$ as \emph{Hecke type}, by analogy with Bump's classification of Rankin--Selberg integrals. In this case, our norm-compatibility machinery gives compatible families of cycle classes (topological or \'etale).

   \subsubsection{Diagonal embeddings of general linear groups}
    \label{sect:GLnGGP}

    Our first example is $G = \GL_n \times_{\Gm} \GL_{n+1}$, for $n \ge 1$, where $\times_{\Gm}$ denotes fibre product over the determinant map to $\Gm$. We take $H = \GL_n$, embedded via $g \mapsto (g \oplus 1, g)$, and $Q_H^0 = H$.

    It is well-known that $(G, H)$ is a spherical pair, i.e.~$H$ has an open orbit on the Borel flag variety of $G$. Moreover, since $\dim H = \dim(G/ B_G) = n^2$, the stabiliser of a point in this orbit has to be finite, and one checks that it is in fact trivial. Hence we can take $Q_G$ to be the upper triangular Borel of $G$; and we can take the abelian quotient $C$ to be the maximal torus quotient of $H$, which is $\GL_1$.

    This case can be globalised in several distinct ways. Firstly, we can take $\cG$ and $\cH$ to be the corresponding split groups over $\QQ$. For $n = 1$ this recovers the classical ``modular symbol'' construction of norm-compatible families of classes in Betti $H^1$ of modular curves, given by paths between cusps, and thus we recover the classical construction of the standard $p$-adic $L$-function of a modular form. For $n > 1$, we obtain $p$-adic $L$-functions associated to Rankin--Selberg $L$-functions for $\GL_n \times \GL_{n+1}$; this recovers a construction due to Kazhdan--Mazur--Schmidt \cite{kms} and Januszewski \cite{januszewski}.

     Alternatively, we can take $\cG$ and $\cH$ to be unitary groups (relative to some imaginary quadratic field $K$ in which $p$ splits), with $\RR$-points
    \[ U(n-1, 1) \into U(n-1, 1) \mathop{\times}_{U(1)} U(n, 1),\]
    and use \'etale cohomology of the canonical models over $K$. The pushforward of the identity class in $H^0_{\et}(\cY_\cH)$ is then the \'etale cycle class of a diagonal cycle, living in the group $H^{2n}_{\et}(\cY_\cG, \Zp(n))$ (the ``arithmetic middle degree''); and the action of Galois on the abelian quotient $\mathcal{C} = U(1)$ cuts out the anticyclotomic extensions of $K$, so we obtain norm-compatibility relations in the anticyclotomic tower at $p$. For $n = 1$, this recovers the norm-compatible family of Heegner points; the case $n = 2$ is essentially the setting of the work of Jetchev and his group on ``unitary diagonal cycles'' \cite{BBJ18}. (The restriction to $p$ split is easily removed, since one checks that the same group-theoretic criterion -- that $H$ have an open orbit on $G/B_G$ with trivial stabiliser -- is also satisfied for the unramified unitary groups associated to a quadratic extension of $\QQ_p$.)

    \begin{remark}
     We have, of course, discarded much useful information by taking $L_G^0$ to be the kernel of the determinant map, when we could in fact have taken it to be $\{1\}$. The extra information arising from choosing $L_G^0 = \{1\}$ is essentially the fact that the above Betti and \'etale classes interpolate in Hida-type $p$-adic families as the weight varies.
    \end{remark}

   \subsubsection{Diagonal embeddings of orthogonal groups}
    \label{sect:SOnGGP}
    We can also consider the analogue of the above construction for orthogonal groups: we choose a quadratic space $V / \QQ$ of dimension $n$, set $V' = V \oplus \QQ e$ where $\langle e, e \rangle = 1$, and define $H = \operatorname{SO}(V)$, $G = \operatorname{SO}(V \oplus V')$. Again, it is well-known that $(G, H)$ is a spherical pair, and it has the property that the stabiliser of a point in the open orbit is trivial. If we choose our global groups in such a way that the picture at $\infty$ is
    \[ \operatorname{SO}(n-2, 2) \into \operatorname{SO}(n-1, 2) \times \operatorname{SO}(n-1, 2), \]
    then we again have a diagonal cycle on a Shimura variety for $\cG$, living in the arithmetic middle degree.

    For $n = 2$ this again recovers Heegner points (up to a harmless central isogeny), since $\operatorname{SO}(2)$ is a non-split torus and $\operatorname{SO}(1, 2) \cong \operatorname{PGL}_2$.  For $n = 3$, the group $\cG$ is isogenous to $\SL_2 \times \SL_2 \times \SL_2$, and we recover the Gross--Kudla--Schoen cycles of \cite{darmonrotger14}.

    \begin{remark}
     For $n \ge 3$ the group $\cG$ has no nontrivial torus quotient. However, the construction can be naturally interpreted in terms of variation in $p$-adic families of Hida type.
    \end{remark}

   \subsubsection{The case of $\GL_n \times \GL_n \subset \GL_{2n}$}
    \label{sect:gl2n}

    We can also consider the ``block--diagonal'' embedding of $H = \GL_n \times \GL_n$ into $G = \GL_{2n}$. In this case, there are two natural possibilities to consider.

    The first option is to take $Q_G$ to be the parabolic having $H$ as its Levi factor. In this case, a straightforward calculation shows that if $u = \stbt{I_n}{I_n}{0_n}{I_n}$, where $I_n$ is the $n \times n$ identity matrix, then $u^{-1} H u \cap \bar{Q}_G = \{ \stbt{X}{}{}X : X \in \GL_n\}$. This is in the kernel of the natural map $H \to \Gm$ given by $(h_1, h_2) \to \det(h_1) / \det(h_2)$, so the construction of \S\ref{sect:abtowers} gives norm-compatibility in an abelian tower with $\Zp^\times$ as the quotient.

    One global setting in which this local computation applies is that studied by \cite{DJR18}: here we take $\cG$ and $\cH$ to be the split groups $\GL_{2n}$ and $\GL_{n} \times \GL_n$, and one obtains $p$-adic $L$-functions for cohomological automorphic representations of $\GL_{2n}(\AA_{\QQ})$ having a Shalika model. However, one can also consider unitary groups split at $p$, with signature at $\infty$ given by
    \[ U(n-1, 1) \times U(n, 0) \into U(2n-1, 1), \]
    in which case one again obtains \'etale classes landing in the arithmetic middle degree, satisfying a norm-compatibility property along the anti-cyclotomic extension of the CM field defining the unitary groups. The properties of these classes are studied further in recent work of Graham and Shah \cite{grahamshah20}, who have shown that these classes also satisfy a ``horizontal'' norm-compatibility relation at primes split in the CM field.

    A more powerful, but more complex, variant of this construction is obtained by taking $B_G$ to be the Borel subgroup, and $u$ the element $\stbt{I_n}{J_n}{0_n}{I_n}$, where $J_n$ is an anti-diagonal matrix of 1's. Then one checks that $u^{-1} H u \cap \bar{Q}_G$ is contained in $L_G$: it is the group of diagonal matrices of the form
    \[ \operatorname{diag}(x_1, \dots, x_n, x_n, \dots, x_1). \]
    So $L_G / L_G^0$ is isomorphic to $\Gm^n$. We can obtain a further $\Gm$ factor by extending our embedding to $\GL_{2n} \times \GL_1$ via the formula $(h_1, h_2) \mapsto \left(\stbt{h_1}{}{}{h_2}, \det h_1\right)$. This can be used to construct $(n+1)$-parameter families of cohomology classes, with one ``abelian'' variable (as above) and $n$ ``weight'' variables. We hope to explore the applications of this construction further in a subsequent paper.

    \begin{remark}
     For $n = 1$, the $\GL_{n} \times \GL_n \into \GL_{2n} \times \GL_1$ and $\GL_n \into \GL_{n+1} \times_{\Gm} \GL_n$ constructions coincide (up to a redundant extra $\GL_1$ factor).
    \end{remark}

  \subsection{Eisenstein-type pairs}

   There are also important examples in which $Q_H^0$ is a proper subgroup of $H$, and we take the input to be one of the non-trivial Iwasawa cohomology classes mentioned in \S\ref{sect:Eisclass}. Since these classes are associated to Eisenstein series for the parabolic $Q_H \subset H$, we refer to these as \emph{Eisenstein type}.

   \subsubsection{Rankin--Selberg and its twists} If we take $H$ to be $\GL_2$, $G$ to be $\GL_2 \times_{\Gm} \GL_2$, and $Q_H^0$ the mirabolic $\stbt \star \star {} 1$, then our conditions are satisfied for the orbit of $u = \stbt 1 1 {} 1$, with $L_G^0 = \{1\}$.

   The ``obvious'' global application, in which $\cG$ and $\cH$ are the split rational forms, is the $p$-direction norm relation for the Euler system of Beilinson--Flach classes \cite{LLZ}. However, taking $\cH = \GL_2$ but $\cG$ to be the quasi-split form of $\GL_2 \times_{\Gm} \GL_2$ corresponding to a real quadratic field $F / \QQ$ (with $p$ split in $F$), then the same local computation gives norm relations for the Asai--Flach Euler system \cite{leiloefflerzerbes18}; and taking $F$ instead to be imaginary quadratic, and our cohomology functors to be Betti rather than \'etale, one obtains the $p$-adic $L$-function of \cite{loefflerwilliams18}.

   \subsubsection{$\GSp_4$ and $\GSp_4 \times \GL_2$}

    We can also take $H = \GL_2 \times_{\Gm} \GL_2$, and $G$ to be $\GSp_4$, with the embedding $\iota: H \into G$ given by a decomposition of the standard representation of $\GSp_4$ into two orthogonal subspaces. We can take $Q_H^0$ to be the fibre product of two copies of the standard mirabolic $\stbt \star \star {} 1$ of $\GL_2$.

    One checks easily that there is an open orbit of $Q_H^0$ in the flag variety for the Siegel parabolic $P_{\mathrm{S}} \subset G$ (the stabiliser of a 2-dimensional isotropic subspace), and the stabiliser of this orbit is $\{1\}$. So we can take $Q_G^0$ to be the unipotent radical of $P_{\mathrm{S}}$. In particular, this group has trivial image under the multiplier map $\mu: G \to \GL_1$, so we deduce a norm-compatibility statement for the abelian tower given by the groups
    \[ J_n = \{ x: x \in P_\mathrm{S} \bmod p, \mu(x) = 1 \bmod p^n\}, \]
    involving the Hecke operator given by the double coset of $\operatorname{diag}(p, p, 1, 1)^{-1}$. This is, of course, precisely the computation underlying the vertical norm-compatibility of the \'etale Lemma--Flach classes of \cite{LSZ17}.

    Another possibility is to take the same group $H$, but to embed it instead into $G = \GSp_4 \times_{\Gm} \GL_2$, with the map being the fibre product of the above map $H \into \GSp_4$ and the second projection $H \to \GL_2$. Here we take $Q_H^0 =  \left(\stbt \star \star {} 1, \star \right)$. One checks that $Q_H^0$ has an open orbit on the Borel flag variety of $G$, with trivial stabiliser. This gives a natural norm-compatible family of \'etale classes in the degree 5 \'etale cohomology of the Shimura variety for $G$ (which is the arithmetic middle degree, since $\cY_\cG$ has dimension 4). This will be investigated in detail in forthcoming work of Hsu, Jin, and Sakamoto, based on a project led by the author and Zerbes at the 2018 Arizona Winter School.

    \begin{remark}
     Iterating the process once more gives a Hecke-type pair
     \[ (G, H) = \Big(\GSp_4 \mathop{\times}_{\Gm} \GL_2 \mathop{\times}_{\Gm} \GL_2,\ \GL_2\mathop{\times}_{\Gm} \GL_2\Big). \]
     After factoring out the copy of $\Gm$ in the centre of both $H$ and $G$, this becomes the $n = 4$ case of the $(\SO_n \times \SO_{n+1}, \SO_n)$ construction above, since $\SO_5 \cong \operatorname{PGSp}_4$ and $\SO_4 = \left(\GL_2 \times_{\Gm} \GL_2\right) / \Gm$. Together with the $\GSp_4$ and $\GSp_4 \times \GL_2$ constructions this gives a new ``trilogy'' of Euler systems, as a sequel to the ``tale of two trilogies'' described in \cite{BCDDPR}.
    \end{remark}

%

   \subsubsection{$\GL_3 \times \GL_1$ and its twists} We can also consider the embedding
   \begin{gather*}
   H = \GL_2 \times \GL_1 \into G = \GL_3 \times \GL_1\\
   (g, x) \mapsto  \left(\stbt g {}{}x, x\right),
   \end{gather*}
   with $Q_H^0 = \stbt \star \star {} 1 \times \GL_1$. One computes that there is an open $Q_H^0$-orbit on the Borel flag variety of $G$ whose stabiliser is trivial; in particular, we obtain norm-compatibility in an abelian tower with Galois group $\Zp^\times \times \Zp^\times$.

   This example has (at least) two interesting global applications. Firstly, we can take $K$ imaginary quadratic (with $p$ split), $\cH$ to be the group $\GL_2 \times_{\Gm} \Res_{K/\QQ} \GL_1$, and $\cG$ the quasi-split unitary group $\GU(2, 1)$. Then the Galois action on the above abelian tower cuts out the maximal abelian extension of $K$ unramified outside $p$, so we obtain norm-compatible families of Galois cohomology classes over this 2-dimensional $p$-adic Lie extension. This will be pursued in forthcoming work of the author with Skinner and Zerbes. Secondly, we can take $\cG$ and $\cH$ to be split over $\QQ$, in which case we obtain a 2-variable measure with values in the Betti cohomology of the 5-dimensional symmetric space for $\GL_3$. This will be studied further in a forthcoming work of the author and Williams, where it will be used to construct a $p$-adic $L$-function for cohomological automorphic representations of $\GL_3 / \QQ$ which are not necessarily self-dual.

   \subsubsection{Some curiosities}

    We mention two further pairs of groups in which case the above machinery seems \emph{not} to work as well as one would hope.

    The preprint \cite{cauchi17} studies the image of the Faltings Eisenstein classes for $H = \GSp_4$ under an embedding into $\GU(2, 2)$, which factors through the kernel $G$ of the natural map $\GU(2, 2) \to \mathrm{U}(1)$; this kernel $G$ is isomorphic to $\operatorname{GSpin}(4, 2)$. The choice of these Eisenstein classes requires us to take $Q_H^0$ to be the 7-dimensional Klingen mirabolic. However, the Borel flag variety of $G$ has dimension 6 (the root system of $G$ is the same as that of $\SL_4$). So every orbit of $Q_H^0$ on the flag variety of $G$ has to have stabilisers of dimension at least 1; and one computes that the stabilisers always surject onto the maximal torus quotient of $G$. Hence there is no way to obtain norm-compatibility in an ``abelian'' tower for these cohomology classes; but one obtains instead a compatibility in Hida-type families. (The situation would be much improved if one could take $Q_H^0$ to be a mirabolic associated to the Borel subgroup of $H$, but we do not know of a construction of Iwasawa cohomology classes for this level tower.)

    The preprint \cite{cauchirodrigues18} studies the embedding
    \[ \GL_2 \mathop{\times}_{\mathbf{G}_m} \GL_2 \mathop{\times}_{\mathbf{G}_m} \GL_2 \into \GSp_6, \]
    taking $Q_H^0$ to be the group $(\stbt \star \star {} 1, \star, \star)$. This group has dimension 8, so it has no chance of having an open orbit on the 9-dimensional Borel flag variety of $G$. Instead, one checks that it has an open orbit on $G / \bar{Q}_G$, where $\bar{Q}_G$ is block-upper-triangular with blocks of size $(1, 2, 2, 1)$, and the stabiliser of a point in this orbit is trivial. Hence one obtains norm-compatibility in a very large $p$-adic tower. The difficulty in this case is in proving the \emph{horizontal} norm relations: the strategy followed for $\GSp_4$ in \cite{LSZ17} relies on a multiplicity-one property for decompositions of spherical representations of $G \times H$, which is closely bound up with the existence of an open orbit of $Q_H$ (sic, not $Q_H^0$) on $G / \overline{B}_G$. This is clearly impossible, for dimension reasons.

    \begin{remark}
     Note that in the first example, the problem is that $Q_H^0$ is ``too large'', and the second example, $Q_H^0$ is ``too small''.
    \end{remark}

 \section{Brion's classification}

  The embeddings of algebraic groups $H \into G$ such that $G$ is semisimple, $H$ is reductive and $(G, H)$ is a spherical pair (i.e. $H$ has an open orbit on $G/B_G$) have been classified by Brion \cite{brion87}. It suffices to classify the associated pairs of Lie algebras $(\mathfrak{g}, \mathfrak{h})$. These are all built up via direct products from an explicit list of ``indecomposable'' pairs.

  In Hecke-type constructions (with $Q_H^0 = H$), our machinery works most neatly if $\dim(G/B_G) = \dim H$, so the stabiliser of a point in the open orbit is finite. There are eight infinite families of indecomposable pairs $(\fg, \fh)$ with this property: the pairs $(\mathfrak{sl}_n \times \mathfrak{sl}_{n+1}, \mathfrak{sl}_n \times \mathfrak{t})$ and $(\mathfrak{so}_n \times \mathfrak{so}_{n+1}, \mathfrak{so}_n)$, which correspond to the constructions of \S \ref{sect:GLnGGP} and \S \ref{sect:SOnGGP} respectively, and six more:
  \begin{multicols}{2}
   \begin{itemize}
    \item $(\mathfrak{sl}_n, \mathfrak{so}_n)$
    \item $(\mathfrak{sl}_{2n+1}, \mathfrak{sp}_{2n})$
    \item $(\mathfrak{so}_{2n+1}, \mathfrak{so}_n \times \mathfrak{so}_{n+1})$
    \item $(\mathfrak{so}_{2n+1}, \fsl_n \times \mathfrak{t})$
    \item $(\mathfrak{sp}_{2n}, \fsl_n \times \mathfrak{t})$
    \item $(\mathfrak{so}_{2n}, \mathfrak{so}_n \times \mathfrak{so}_n)$
   \end{itemize}
  \end{multicols}
  Here $\mathfrak{t}$ is the 1-dimensional abelian Lie algebra. There are also the following ``sporadic'' examples:
  \begin{multicols}{2}
   \begin{itemize}
    \item $\Big((\fsp_4)^2 \times \fsl_2, (\fsl_2)^3\Big)$
    \item $\Big((\fsp_4)^3, (\fsl_2)^4\Big)$
    \item $\Big(\fsp_6 \times \fsp_4, \fsp_4 \times \fsl_2\Big)$
    \item $\Big(\fsp_8 \times \fsp_4, \fsp_4 \times \fsp_4\Big)$
    \item $\Big(\fsl_3 \times \fsp_4, (\fsl_2)^2 \times \mathfrak{t}\Big)$
    \item $\Big(\fsl_4 \times \fsl_2, (\fsl_2)^2 \times \mathfrak{t}\Big)$
    \item $\Big(\fsl_4 \times \fsp_4, (\fsl_2)^3 \times \mathfrak{t}\Big)$
    \item $\Big(\mathfrak{e}_6, \fsp_8\Big)$
    \item $\Big(\mathfrak{e}_7, \fsl_8\Big)$
    \item $\Big(\mathfrak{e}_8, \so_{16}\Big)$
    \item $\Big(\mathfrak{f}_4, \fsp_6 \times \fsl_2\Big)$
    \item $\Big(\mathfrak{g}_2, \fsl_2 \times \fsl_2\Big)$
   \end{itemize}
  \end{multicols}
  The cases in which $\mathfrak{h}$ has $\mathfrak{t}$ as a factor are particularly interesting, because we can potentially use this to obtain norm-compatibility in a non-trivial abelian level tower. One of these, the pair $(\mathfrak{so}_{2n+1}, \fsl_n \times \mathfrak{t})$, appears in recent work of Cornut \cite{cornut}, who constructs an Euler system in the cohomology of a Shimura variety for $\SO(2n-1, 2)$, using cycles given by an embedding of $U(n-1, 1)$; our theory thus allows Cornut's Euler system to be extended over the $p$-adic anticyclotomic tower. Arithmetic applications of the remaining cases, such as $\fsl_n\times \mathfrak{t} \into \fsp_{2n}$, do not seem to have been explored at all (beyond the case $n = 1$) and it would be highly interesting to do so.

  This list can also be used to find new examples of ``Eisenstein-type'' constructions, by searching for spherical pairs $(\mathfrak{g}, \fh)$ such that $\fg = \fg' \times (\fsl_2)^n$, $\fh = \fh' \times (\fsl_2)^n$ for some $\fg', \fh'$ and $n \ge 1$, and the map $\fh \into \fg$ identifies the two $(\fsl_2)^n$ factors. This gives rise to spherical pairs of the form $(Q^0)^n \times H' \into G'$, where $Q^0$ is the standard mirabolic in $\GL_2$; and we can obtain norm-compatible families in symmetric spaces for $G'$ by pushing forward $\GL_2$ Eisenstein classes. Searching for pairs $(\fg, \fh)$ of this form in the above list, we recover the constructions described above for $G' = \GL_2 \times_{\Gm} \GL_2$, $\GSp_4$, $\GSp_4 \times_{\Gm} \GL_2$, and $\GL_3 \times \GL_1$, and two new cases, namely $G' = \GL_4$ and $G' = \GSp_4 \times_{\Gm} \GSp_4$.

  \begin{remark}
   This list does not exhaust the potential applications of our main theorem, since it only covers cases where we can take $Q_G$ to be the Borel of $G$ and $L_G^0 = \{1\}$. These are not the only cases where the method gives interesting results, as the examples of \S \ref{sect:gl2n} show.
  \end{remark}

\providecommand{\bysame}{\leavevmode\hbox to3em{\hrulefill}\thinspace}
\providecommand{\MR}[1]{}
\renewcommand{\MR}[1]{%
}
\providecommand{\href}[2]{#2}
\newcommand{\articlehref}[2]{\href{#1}{#2}}

\end{document}